\documentclass[11pt,a4paper]{amsart} 
\usepackage{amsmath,amsfonts,amssymb,amsthm}
\usepackage{graphicx}
\usepackage[hidelinks]{hyperref}
\usepackage{bbm}
\usepackage[foot]{amsaddr}

\calclayout

\newtheorem{assumption}{Assumption}
\newtheorem{defn}{Definition}
\newtheorem{theorem}{Theorem}

\newtheorem{lemma}{Lemma}

\usepackage{tikz}
\usetikzlibrary{decorations.pathreplacing,angles,quotes}

\usepackage{mathtools}

\DeclarePairedDelimiter\floor{\lfloor}{\rfloor}

\def\11{{\mathbbm 1}}
\def\S{{\mathcal S}}
\def\EE{{\mathbb E}}
\def\GGG{{\mathbf G}}
\def\BBB{{\mathbf B}}

\def\XXX{{\mathbf X}}
\def\ZZZ{{\mathbf Z}}

\def\AAA{{\mathbf A}}
\def\KKK{{\mathbf K}}
\def\WWW{{\mathbf W}}
\def\PPP{{\mathbf P}}

\def\RR{{\mathbb R}}

\def\PP{\mathbb{P}}
\def\NN{\mathbb{N}}
\def\BBBB{\mathfrak{B}}

\title{Value iteration for approximate dynamic programming under
  convexity}

\author{Jeremy Yee}
\email{jeremyyee@outlook.com.au}
\date{\today}

\begin{document}
\maketitle

\begin{abstract}
  This paper studies value iteration for infinite horizon contracting
  Markov decision processes under convexity assumptions and when the
  state space is uncountable. The original value iteration is replaced
  with a more tractable form and the fixed points from the modified
  Bellman operators will be shown to converge uniformly on compacts
  sets to their original counterparts. This holds under various
  sampling approaches for the random disturbances.  Moreover, this
  paper will present conditions in which these fixed points form
  monotone sequences of lower bounding or upper bounding functions for
  the original fixed point. This approach is then demonstrated
  numerically on a perpetual Bermudan put option.
\end{abstract}

\smallskip
\noindent \textbf{Keywords.}  Convexity, Infinite horizon, Markov
decision processes, Value iteration


\section{Introduction}
\label{sec_intro}

Infinite horizon discounted Markov decision processes occur frequently
in practical applications and so the literature devoted to this class
of problems is already well developed. This paper refers reader to the
textbooks \cite{howard, hinderer, puterman,
  hernandezlerma_lasserreBook, powell, bauerle_rieder} and the great
many references cited within them. Typical numerical approaches to
these problems include policy iteration or value iteration. This paper
focuses on the value iteration procedure for uncountable state spaces
and finite action spaces. Value iteration has been used extensively
over the decades and early works include
\cite{shapley1953,blackwell1965,denardo1967}. The following gives a
brief intuition of this approach and the issues that arises in our
setting. A more formal description will be given in the next
section. Suppose the Bellman operator $\mathcal{T}$ is given by
$$
({\mathcal T}v)(x)=\max_{a \in \AAA} (r(x,a)+ \beta{\mathcal K}^{a}v(x))
$$
where $r$ is the reward function, $a$ is the action, $x$ is the state,
$\beta$ is the discount factor, and $\mathcal{K}^{a}$ is the
transition operator which gives the expected value functions. Value
iteration relies on the Bellman operator being a contraction mapping
so that successive applications of the the above operator converges to
a fixed point. However, there are two main issues in numerical work:
\begin{itemize}
\item If the random variable driving $\mathcal{K}^{a}$ takes an
  uncountable number of possible values, one often needs a discretized
  version of $\mathcal{K}^{a}$ for numerical tractability, and
\item If the state $x$ takes a uncountable number of values, one may
  need to represent the reward and expected value functions using
  numerical tractable objects.
\end{itemize}

It turns out that under certain convexity assumptions, the above
problems can be addressed. This is the main contribution of this
paper. In this paper, the original problem is replaced with a more
tractable approximation and the new fixed points are shown to converge
uniformly on compact sets to the original solution. Under additional
assumptions, the fixed points from the modified value iteration form a
non-decreasing sequence of lower bounding functions or a
non-increasing sequence of upper bounding functions. In this way, they
can provide error bounds for the approximation scheme. This paper is
best seen as an extension of \cite{yee_convex} where the author used
convex function approximations for finite horizon Markov decision
processes. This paper is organized as follows. In the next section,
the problem setting and the convexity assumptions are
presented. Section \ref{sec_iteration} then presents the modified
value iteration which exploits these convexity assumptions. The
convergence properties of this approach is then studied in Section
\ref{sec_convergence}. The lower and upper bounding behaviour of the
fixed points from the modified value iteration are derived in Section
\ref{secLower} and Section \ref{secUpper}. These methods are then
demonstarted numerically on a perpetual Bermudan put option in Section
\ref{sec_numerical}. Finally, Section \ref{sec_conclusion} concludes
this paper.


\section{Problem setting}
\label{sec_problem}

The folllowing introduces the infinite horizon discounted Markov
decision process.  Denote time by $t\in\NN$ and the state by
$X_t := (P_t, Z_t)$ consisting of a discrete component $P_t$ taking
values in some finite set $\PPP$ and a continuous component $Z_t$
taking values in an open convex set $\ZZZ \subseteq \RR^d$. The set
$\XXX = \PPP\times\ZZZ$ is referred to as the state space. At each
time $t\in\NN$, an action $a \in \AAA$ is chosen by the agent from a
finite set $\AAA$. Suppose the starting state is given by $X_0 = x_0$
almost surely. Assume the discrete component evolves as a controlled
finite state Markov chain with transition probabilities
$\alpha^a_{p,p'}$ for $a\in\AAA$, and $p,p'\in\PPP$. The value
$\alpha_{p,p'}^a$ gives the probability of transitioning from $P_t=p$
to $P_{t+1}=p'$ after action $a$ is applied at time $t$. Suppose
action $a$ governs the evolution of the continuous component via
\begin{equation*}
Z_{t+1} = f(W^a_{t+1}, Z_t)
\end{equation*}
for random variable $W^a_{t+1}: \Omega \to \WWW \subseteq \RR^{d'}$
and measurable transition function
$f:\WWW \times \ZZZ \rightarrow \ZZZ$.  The variable $W^a_{t+1}$ is
referred to as the random disturbance driving the evolution of the
state after action $a$ is chosen. These random variables are assumed
to be identically and independently distributed across time and
actions. The random variables $W^a_{t+1}$ and $f(W^a_{t+1},z)$ are
assumed to be integrable for all $z\in\ZZZ$, $t\in\NN$, and
$a\in\AAA$.

Decision rule $\pi_{t}$ gives a mapping $\pi_{t}: \XXX \to \AAA$ which
allocates an action $\pi_{t}(x) \in \mathbf{A}$ for given state
$x \in \XXX$. A policy is defined as a sequence of decision rules
i.e. $\pi = (\pi_{t})_{t\in\NN}$. For each starting state
$x_{0} \in \XXX$, each policy $\pi$ induces a probability measure such
that $\PP^{x_{0},\pi}(X_{0}=x_{0})=1$ and
\begin{equation*}
\PP^{x_{0}, \pi}(X_{t+1} \in \BBB \, | \, X_{0}, \dots, X_{t})=K^{\pi_{t}(X_{t})}(X_{t},\BBB )
\end{equation*}
for each measurable $\BBB \subset \XXX$ where $K^{a}$ denotes our
Markov transition kernel after applying action $a$. The reward at all
times $t\in\NN$ is given by the time homogenous function
$r: \XXX \times \AAA \to \RR$. The controller's aim is to maximize the
expectation
\begin{equation}
\label{policyvalue}
v^\pi(x_0) =\EE^{x_{0}, \pi}\left(\sum_{t=0}^{\infty} \beta^t r_{t}(X_{t}, \pi_{t}(X_{t})) \right)
\end{equation}
over all possible policies $\pi$ where $0<\beta<1$ is known as the
discount rate. A policy $\pi^{*}=(\pi^{*}_{t})_{t\in\NN}$ is said to
be optimal if it satisfies $v^{\pi^*}(x_0) \geq v^{\pi}(x_0)$ for any
possible policy $\pi$ and $x_0\in\XXX$. The function $v^{\pi^*}$ is
often referred to as the \emph{value function}.

\begin{assumption} \label{assContract}
  There exists function $b:\PPP\times\ZZZ\to\RR_+$ with constants
  $c_r, c_b \in\RR_+$ such that
  $$
  |r(p,z,a)| \leq c_rb(p,z) \quad \text{ and } \quad \int_{\XXX} b(p',z') K^a((p,z), \mathrm{d} (p',z')) \leq c_bb(p,z) 
  $$
  for $p\in\PPP$, $z\in\ZZZ$, and $a\in\AAA$. Assume $\beta c_b < 1$
  and that $b(p,z)$ is continuous in $z$ for all $p\in\PPP$.
\end{assumption}

With the above assumption, the decision problem now becomes a
so-called \emph{contracting Markov decision process}. To reveal why,
the following is needed. Define the weighted supremum norm
$\|\cdot\|_b$ of a function $v:\PPP\times\ZZZ\to\RR$ by
$$
\|v\|_b := \sup_{p\in\PPP,z\in\ZZZ} \frac{|v(p,z)|}{b(p,z)}
$$
and denote $\BBBB_b$ to be the family of all measurable functions
$v:\PPP\times\ZZZ\to\RR$ such that $\|v\|_b < \infty$. Let $W^a:\Omega\to\WWW$
be an identical distributed and independent copy of
$(W^a_{t})_{t\in\NN}$ and define ${\mathcal K}^{a}$ to represent the
one step transition operator which acts on functions $v$ by
\begin{equation}
  ({\mathcal K}^{a}v)(p,z) = \sum_{p'\in\PPP} \alpha^a_{p,p'}\EE[v(p',f(W^a,z))]
  \label{transoperator}
\end{equation}
whenever the expectations are well defined. Introduce the Bellman
operator $\mathcal{T}:\BBBB_b\to\BBBB_b$ by
\begin{equation}
({\mathcal T}v)(p,z)=\max_{a \in \AAA} (r(p,z,a)+ \beta{\mathcal K}^{a}v(p,z))
\label{1bel}
\end{equation}
for $p\in\PPP$ and $z\in\ZZZ$. The following theorem establishes the
Bellman operator as a contraction mapping on $(\BBBB_b, \|\cdot\|_b)$.

\begin{lemma} \label{lemmaContraction} It holds that
  $ \|\mathcal{T} v' - \mathcal{T}v'' \|_b \leq \beta c_b \| v' -
  v''\|_b $ for all $v',v''\in\BBBB_b$.
\end{lemma}
\begin{proof}
  See \cite[Lemma 7.3.3]{bauerle_rieder}.
\end{proof}

\begin{theorem} \label{theoremFixed}
  Let $v^*\in\BBBB_b$ be the unique fixed point of
  $\mathcal{T}$. If decision rule $\tilde \pi_0 $ satisfies
  $$
  v^*(p,z)=  r(p,z,\tilde\pi_0(p,z))+ \beta{\mathcal K}^{\tilde\pi_0(p,z)}v^*(p,z)
  $$
  for all $p\in\PPP$ and $z\in\ZZZ$, then $v^{\pi^*} = v^*$ and
  $\pi^* = (\tilde\pi_0,\tilde\pi_0,\tilde\pi_0, \dots)$.
\end{theorem}
\begin{proof}
  See \cite[Theorem 7.3.4]{bauerle_rieder}.
\end{proof}

The above shows that an optimal policy is stationary and that the
value function can be found through the well known approach of
\emph{value iteration}. In the following, some convexity assumptions
are imposed.

\begin{assumption} \label{assConvex} Let $r(p,z,a)$ be convex in $z$
  for all $a \in \AAA$ and $p \in \PPP$.  If $h(p,z)$ is convex in $z$
  for $p\in\PPP$, then $\mathcal{K}^ah(p,z)$ is also convex in $z$ for
  $a\in\AAA$ and $p\in\PPP$.
\end{assumption}

The above assumption ensures that the value function is convex in $z$
for all $p\in\PPP$.

\begin{lemma} \label{lemmaUC} Let $v_n,\widehat v \in\BBBB_b$ for
  $n\in\NN$. It holds that $\lim_{n\to\infty}\|v_n-\widehat v\|_b = 0$
  $\iff $ $v_n$ converges to $\widehat v$ uniformly on compact sets.
\end{lemma}
\begin{proof}
  Recall $b(p,z)$ is continuous in $z$ and positive for $p\in\PPP$. On
  compact set $\KKK \subset \PPP\times\ZZZ$,
  $\bar b_\KKK := \sup_{p\in\PPP,z\in\ZZZ}b(p,z) < \infty$ and
  $\underline b_\KKK := \inf_{p\in\PPP,z\in\ZZZ}b(p,z) < \infty$. It
  can be seen that
  $$
  \sup_{(p,z)\in\KKK}|v_n(p,z) - \widehat v(p,z)| \leq \bar b_\KKK
  \sup_{(p,z)\in\KKK}\frac{|v_n(p,z) - \widehat v(p,z)|}{b(p,z)} \leq
  \frac{\bar b_\KKK}{\underline b_\KKK} \sup_{(p,z)\in\KKK} |v_n(p,z)
  - \widehat v(p,z)|.
  $$
\end{proof}

\begin{theorem} \label{valueConvex}
  The value function $v^{\pi^*}(p,z)$ is convex in $z$ for all
  $p\in\PPP$.
\end{theorem}
\begin{proof}
  Let $v_0(p, z)$ be a function convex in $z$ for all $p\in\PPP$ and
  $v_0\in\BBBB_b$. Observe that
  $({\mathcal T}v_0)(p,z)=\max_{a \in \AAA} (r(p,z,a)+ {\mathcal
    K}^{a}v_0(p,z))$ is convex in $z$ for $p\in\PPP$ since the sum and
  pointwise maximum of a finite number of convex functions is
  convex. Therefore, by induction, successive value iteration always
  results in a function convex in $z$ for $p\in\PPP$. Let
  $[\mathcal{T}]^i$ represent $i\in\NN$ successive applications of the
  operator $\mathcal{T}$. Since $v^{\pi^*}$ is the unique fixed point,
  it can be shown that
  $$
  \lim_{i\to\infty} \|v^{\pi^*}-[\mathcal{T}]^i v_0\|_b = 0
  $$
  and so by Lemma \ref{lemmaUC}, $[\mathcal{T}]^i v_0$ converges to
  $v^{\pi^*}$ uniformly on compact sets. Since
  $[\mathcal{T}]^i v_0(p,z)$ is convex in $z$ for $i\in\NN$ and
  $p\in\PPP$, $v^{\pi^*}(p,z)$ is also convex in $z$ for $p\in\PPP$
  since convexity is preserved under pointwise convergence.
\end{proof}

The following continuity assumption is imposed on the transition
function and is needed for the convergence results presented later on.

\begin{assumption} \label{assContinuity} Assume $f(w,z)$ is continuous
  in $w$ for all $z\in\ZZZ$.
\end{assumption}


\section{Modified value iteration}
\label{sec_iteration}

This section approximates the value iteration in \eqref{1bel} with a
more tractable form. Let us first approximate the transition operator
\eqref{transoperator}. For each action $a \in \AAA$, choose a suitable
$n$-point disturbance sampling $(W^{a,(n)}(k))_{k=1}^{n}$ with weights
$(\rho^{a,(n)}(k))_{k=1}^{n}$.  Define the modified transition
operator by
\begin{equation}
{\mathcal K}^{a, (n)}v(p,z) = \sum_{p'\in\PPP}\alpha^a_{p,p'} \sum_{k=1}^{n}\rho^{a,(n)}(k) v(p', f(W^{a,(n)}(k),z)).
\label{modKernel}
\end{equation}
Since Assumption \ref{assConvex} ensures that
${\mathcal K}^{a}v(p, z)$ is convex in $z$ if $v(p, z)$ is convex in
$z$, it is only natural to make the same assumption for the above
modified transition operator.

\begin{assumption} \label{assSampleConvex} If $v(p, z)$ is convex in
  $z$ for $p\in\PPP$, assume ${\mathcal K}^{a, (n)}v(p, z)$ is convex
  in $z$ for all $n\in\NN$, $a \in \AAA$ and $p \in \PPP$.
\end{assumption}

Now denote $\mathbf{G}_{}^{(m)}\subset\ZZZ$ to be a $m$-point
grid. Assume that $\mathbf{G}_{}^{(m)} \subset \mathbf{G}_{}^{(m+1)}$
and $\cup_{m=1}^\infty \mathbf{G}_{}^{(m)}$ is dense in
$\ZZZ$. Suppose $h:\ZZZ \rightarrow \RR$ is a convex function and
introduce some approximation scheme ${\mathcal S}_{\mathbf{G}^{(m)}}$
that approximates $h$ using another more tractable convex
function. With this, define the modified Bellman operator by
\begin{equation} {\mathcal T}^{(m,n)}v(p, z) = \max_{a \in
    \AAA}\left( \mathcal{S}_{\mathbf{G}^{(m)}}{r(p, z, a)} +
    \beta \mathcal{S}_{\mathbf{G}^{(m)}}{\mathcal K}^{a, (n)}v(p, z)\right).
\label{modBell}
\end{equation} 
where $v(p, z)$ is convex in $z$ for all $p\in\PPP$. The function
approximation scheme ${{\mathcal S}_{\mathbf{G}^{(m)}}}$ is applied to
the functions for each $p\in\PPP$ and $a\in\AAA$ above. The following
assumption is used to ensure that \eqref{modBell} preserves convexity
for all $p\in\PPP$ and to use for the other theorems presented later
on.

\begin{assumption} \label{piecewiseS} Let $h,h':\ZZZ\to\RR$ be convex
  functions and assume for $m\in\NN$, $a\in\AAA$, and $t\in\NN$ that:
  \begin{itemize}
  \item ${\mathcal S}_{\mathbf{G}^{(m)}}h(z)$ is convex in $z$,
  \item
    ${\mathcal S}_{\mathbf{G}^{(m)}}(c h(z)+h'(z)) = c {\mathcal
      S}_{\mathbf{G}^{(m)}}h(z) + {\mathcal
      S}_{\mathbf{G}^{(m)}}h'(z)$ for $z\in\ZZZ$ and $c\in\RR$, and
  \item $\lim_{m\to\infty} {\mathcal S}_{\mathbf{G}^{(m)}}h(z) = h(z)$
    for all $z\in\cup_{m=1}^\infty \mathbf{G}_{}^{(m)}$.
  \end{itemize}
\end{assumption}

\begin{theorem} \label{successiveConvex} If $v_0(p, z)$ is convex in
  $z$ for $p\in\PPP$, then any successive applications of the modified
  Bellman operator i.e.
  ${\mathcal T}^{(m,n)}\dots {\mathcal T}^{(m,n)}v_0(p, z)$ results in
  a function that is convex in $z$ for all $p\in\PPP$ and $m,n\in\NN$.
\end{theorem}
\begin{proof}
  By assumption, $v_0(p, z)$ is convex in $z$ for $p\in\PPP$. Now
  $\beta \mathcal{S}_{\mathbf{G}^{(m)}}{\mathcal K}^{a, (n)}v(p, z)$
  is convex in $z$ for all $m,n\in\NN$, $p\in\PPP$, and $a\in\AAA$.
  This is a consequence of Assumption \ref{assSampleConvex} and
  Assumption \ref{piecewiseS}. By Assumption \ref{assConvex} and
  Assumption \ref{piecewiseS},
  $\mathcal{S}_{\mathbf{G}^{(m)}}{r(p, z, a)}$ is convex in $z$ for
  all $m\in\NN$, $p\in\PPP$, and $a\in\AAA$. Thus,
  ${\mathcal T}^{(m,n)}v_0(p, z)$ is convex in $z$ since the sum and
  pointwise maximum of a finite number of convex functions is also
  convex. Proceeding inductively for each successive application of
  the modified Bellman operator proves the desired result.
\end{proof}

The next assumption ensures that the modified Bellman operator
represents a contraction mapping on the Banach space
$(\BBBB_b,\|\cdot\|_b)$.  With this, there is a unique fixed point in
$\BBBB_b$ for the modified value iteration represented by
$\mathcal{T}^{(m,n)}$.

\begin{assumption} \label{assModContract}
  Assume for all $m,n\in\NN$ that
  $$
  |\mathcal{S}_{\GGG^{(m)}}r(p,z,a)| \leq c_rb(p,z), \quad
  \quad \mathcal{K}^{a,(n)} b(p,z) \leq
  c_bb(p,z), \quad \text{and} \quad \beta c_b\|\mathcal{S}_{\GGG^{(m)}}b\|_b \leq 1
  $$
  for $p\in\PPP$, $z\in\ZZZ$, and $a\in\AAA$.
\end{assumption}

\begin{theorem} \label{theoremModifiedContraction} It holds that
  $$ \|\mathcal{T}^{(m,n)} v' - \mathcal{T}^{(m,n)}v'' \|_b \leq \beta
  c_b \|\mathcal{S}_{\GGG^{(m)}}b\|_b\| v' - v''\|_b$$ for all
  $v',v''\in\BBBB_b$.
\end{theorem}
\begin{proof}
  \begin{align*}
    \mathcal{T}^{(m,n)} v'(p,z) - \mathcal{T}^{(m,n)}v''(p,z)  
    & \leq \beta \max_{a\in\AAA} \left( \mathcal{S}_{\GGG^{(m)}}{\mathcal K}^{a, (n)}v'(p,z) - \mathcal{S}_{\GGG^{(m)}}{\mathcal K}^{a, (n)}v''(p,z)  \right)\\
    & \leq \beta \max_{a\in\AAA} \left( \mathcal{S}_{\GGG^{(m)}}{\mathcal K}^{a, (n)}(v'(p,z) - v''(p,z))  \right) \\
    & \leq \beta \max_{a\in\AAA} \left( \mathcal{S}_{\GGG^{(m)}}{\mathcal K}^{a, (n)}\left(\frac{v'(p,z) - v''(p,z)}{b(p,z)} b(p,z)\right)  \right) \\
    & \leq c_b \beta \|v' - v''\|_b \mathcal{S}_{\GGG^{(m)}}b(p,z).
  \end{align*}
  The first inequality above follows from
  $\max_{a\in\AAA} l(a) - \max_{a\in\AAA} l'(a) \leq \max_{a\in\AAA}
  l(a) - l'(a)$ and the rest follows from Assumption \ref{piecewiseS}
  and the definition of $\|\cdot\|_b$. Using the same argument above,
  it can be shown that
  $$
  \mathcal{T}^{(m,n)} v''(p,z) - \mathcal{T}^{(m,n)}v'(p,z) \leq c_b
  \beta \|v'' - v'\|_b \mathcal{S}_{\GGG^{(m)}}b(p,z). 
  $$
  Combining the two inequalities above gives
  $$
  \|\mathcal{T}^{(m,n)} v' - \mathcal{T}^{(m,n)}v'' \|_b \leq \beta
  c_b \|\mathcal{S}_{\GGG^{(m)}}b\|_b \| v' - v''\|_b .
  $$
\end{proof}

\begin{theorem} \label{convexApprox} The unique fixed point
  $v^{(m,n),*}(p,z)$ of the modified Bellman operator
  $\mathcal{T}^{(m,n)}$ is convex in $z$ for all $p\in\PPP$,
  $m,n\in\NN$.
\end{theorem}
\begin{proof}
  The existence and uniquess of the fixed point follows from Theorem
  \ref{theoremModifiedContraction} and the Banach fixed point
  theorem. The convexity can be proved in a similar manner as in the
  proof of Theorem \ref{valueConvex} using Theorem
  \ref{successiveConvex}.
\end{proof}

Before proceeding, note that under the assumptions presented so far,
the modified Bellman operator is not necessarilly monotone. Additional
assumptions are need on $\mathcal{S}_{\GGG^{(m)}}$ for this. It turns
out that the convergence results in the next section do not require
this property. However, it is needed for the construction of the lower
or upper bounding functions and so will be imposed then.


\section{Convergence}
\label{sec_convergence}

This section proves that the fixed point in the modified value
iteration represented by \eqref{modBell} converges to the value
function uniformly on compact sets under different sampling schemes
for the random disturbance $W^a$.

\subsection{Disturbance sampling} \label{sectrans}

The below concept of CCC sequences is crucial in this section.

\begin{defn}
  Let $(h^{(n)})_{n\in\NN}$ be a sequence of real-valued convex
  functions on $\ZZZ$ i.e. $h^{(n)}:\ZZZ\to\RR$ for $n\in\NN$. The
  sequence $(h^{(n)})_{n\in\NN}$ is called a CCC (convex compactly
  converging) sequence in $z$ if $(h^{(n)})_{n\in\NN}$ converges
  uniformly on all compact subsets of $\ZZZ$.
\end{defn}

In the following, $(v^{(n)}(p,z))_{n\in\NN}$ is assumed to be a
sequence of functions convex in $z$ for $p\in\PPP$ and $n\in\NN$. So
by Assumption \ref{assContinuity}, $v^{(n)}(p, f(w,z))$ is also
continuous in $w$ for all $p\in\PPP$, $z\in\ZZZ$, and $n\in\NN$ since
we have a composition of continuous functions.  Moreover, suppose this
sequence converges uniformly to function $\widehat v \in\BBBB_b$ on all
compact subsets of $\XXX$. Thus, $(v^{(n)}(p,z))_{n\in\NN}$ is a $CCC$
sequence in $z$ converging to $\widehat v(p,z)$ for $p\in\PPP$. Lemma
\ref{transConvergeMC} below considers the use of random Monte Carlo
samples for the disturbance sampling in the approximate value
iteration. Let the choice of $p\in\PPP$ and $a\in\AAA$ below be
arbitrary.

\begin{lemma} \label{transConvergeMC} Let $(W^{a,(n)}(k))_{k=1}^{n}$
  be a collection of idependently and identically distributed copies
  of $W^a$ and $\rho^{a,(n)}(k) = \frac{1}{n}$ for
  $k=1,\dots,n$. Assume these random variables lie on the same
  probability space as $W^a$. If $\WWW$ is compact, then
  \begin{equation*} 
    \lim_{n\to\infty}
    \mathcal{K}^{a,(n)}v^{(n)}(p,z) =
    \mathcal{K}^{a}\widehat v(p,z), \quad z\in\ZZZ.
  \end{equation*}
  If $\mathcal{K}^{a,(n)}v^{(n)}(p,z)$ is also convex in $z$ for all
  $n\in\NN$, then
  $\left(\mathcal{K}^{a,(n)}v^{(n)}(p,z)\right)_{n\in\NN}$ also forms
  a CCC sequence in $z$.
\end{lemma}
\begin{proof}
  See \cite[Theorem 3]{yee_convex}.
\end{proof}

Note that while the above assumes $\WWW$ to be compact, it does not
place any restriction on the size of $\WWW$. So for unbounded $\WWW$
cases, one can find a compact subset $\overline\WWW\subset\WWW$ that
contains the vast majority of the probability mass. In this sense,
this compactness assumption on $\WWW$ is not restrictive from a
numerical point of view. The above convergence when $\WWW$ is not
compact will be examined in future research. Now while the use of
Monte Carlo sampling may be easier to implement, the user may desire
finer control of the disturbance sampling. The following lemmas may be
useful in this regard.  Introduce partition of the disturbance space
$\WWW$ by
$ \Pi^{(n)}=\{\Pi^{(n)}(k)\subset \WWW \, :\, k=1, \dots, n \}$.  The
situation where $\WWW$ is compact is examined first.

\begin{lemma} \label{transConvergeBounded} Suppose $\WWW$ is
  compact. Denote the diameter of the partition by
  $$
  \delta^{(n)} := \max_{k=1,\dots,n} \sup\{\|w' - w''\| : w',w'' \in
  \Pi^{(n)}(k)\}
  $$
  and let $\lim_{n\to\infty} \delta^{(n)} = 0$.  Choose sampling
  $(W^{a,(n)}(k))_{k=1}^{n}$ where $W^{a,(n)}(k) \in \Pi^{(n)} (k)$
  and $\rho^{a,(n)}(k) = \PP(W^{a} \in \Pi^{(n)}(k))$ for
  $k=1,\dots,n$. Then
\begin{equation*}
  \lim_{n\to\infty} \mathcal{K}^{a,(n)}v^{(n)}(p,z) = \mathcal{K}^{a}\widehat v(p,z), \quad z\in\ZZZ. 
\end{equation*}
If $\mathcal{K}^{a,(n)}v^{(n)}(p,z)$ is also convex in $z$ for all
$n\in\NN$, then
$\left(\mathcal{K}^{a,(n)}v^{(n)}(p,z)\right)_{n\in\NN}$ also forms a
CCC sequence in $z$.
\end{lemma}
\begin{proof} 
  See \cite[Theorem 4]{yee_convex}.
\end{proof}

For the case where $\WWW$ is not compact, the following may be used
instead. The next lemma uses local averages on each component of the
partition. Later on, it will be shown that the use of local averages
minimizes the mean square error from the discreitzation of $W^a$.  To
avoid any possible misunderstandings,
$\EE[W^{a} \mid W^{a} \in \Pi^{(n)}(k)]$ refers to the expectation of
$W^{a}$ conditioned on the event $\{ W^{a} \in \Pi^{(n)}(k)\}$.

\begin{lemma} \label{transConvergeLA} Suppose generated sigma-algebras
  $ \sigma^{(n)}_{a}=\sigma(\{W^a \in \Pi^{(n)}(k)\}, \enspace k=1,
  \dots, n)$ satisfy
  $\sigma(W^a) = \sigma(\cup_{n\in\NN}\sigma^{(n)}_{a})$. Choose
  sampling $(W^{a,(n)}(k))_{k=1}^{n}$ such that
  \[ W^{a,(n)}(k) = \EE[W^{a} \mid W^{a} \in \Pi^{(n)}(k)]\] with
  $\rho^{a,(n)}(k) = \PP(W^{a} \in \Pi^{(n)}(k))$ for
  $k=1,\dots,n$. If
  $\left(v^{(n)}(p', f(W^{a,(n)}, z))\right)_{n\in\NN}$ is uniformly
  integrable for $p'\in\PPP$ and $z\in\ZZZ$, then:
\begin{equation*}
  \lim_{n\to\infty} \mathcal{K}^{a,(n)}v^{(n)}(p,z) = \mathcal{K}^{a}\widehat v(p,z), \quad z\in\ZZZ.
\end{equation*}
If $\mathcal{K}^{a,(n)}v^{(n)}(p,z)$ is also convex in $z$ for all
$n\in\NN$, then
$\left(\mathcal{K}^{a,(n)}v^{(n)}(p,z)\right)_{n\in\NN}$ also forms a CCC
sequence in $z$.
\end{lemma}
\begin{proof} 
  See \cite[Theorem 5]{yee_convex}.
\end{proof}

In the above lemma, the assumption that
$\left(v^{(n)}(p', f(W^{a,(n)}, z))\right)_{n\in\NN}$ is uniformly
integrable for all $p'\in\PPP$, and $z\in\ZZZ$ is
used. The next lemma points to a simple case where this holds and this
case occurs, for example, if the bounding function is constant or when
$\WWW$ is compact.

\begin{lemma}
  Suppose $\sup_{n\in\NN} \| v^{(n)}\|_b < \infty$. If
  $(b(p', f(W^{a,(n)},z')))_{n\in\NN}$ is uniformly integrable, then
  $\left(v^{(n)}(p', f(W^{a,(n)}, z'))\right)_{n\in\NN}$ is uniformly
  integrable for $p'\in\PPP$ and $z'\in\ZZZ$.
\end{lemma}
\begin{proof}
  There exists a positive constant $c$ such that
  $$ 
  \sup_{n\in\NN} \sup_{p\in\PPP,z\in\ZZZ}\frac{|v^{(n)}(p,z)|}{b(p,z)} < c \implies
  \sup_{n\in\NN} \frac{|v^{(n)}(p', f(W^{a,(n)},z'))|}{{c}b(p',
    f(W^{a,(n)},z'))} < 1
  $$ holds with probability one. Thus, $v^{(n)}(p', f(W^{a,(n)}, z')$ is
  dominated by a family of uniformly integrable random variables and
  so is also uniformly integrable.
\end{proof}

\subsection{Choice of sampling}

There are two sources of approximation error induced by the modified
Bellman operator. The first originates from the function approximation
$\mathcal{S}_{\GGG^{(m)}}$ and the other stems from the discretization
of the disturbance $W^a$ using sampling $(W^{a,(n)}(k))_{k=1}^{n}$
with weights $(\rho^{a,(n)}(k))_{k=1}^{n}$. This subsection will
briefly examine the latter where Lemma \ref{transConvergeMC}, Lemma
\ref{transConvergeBounded}, and Lemma \ref{transConvergeLA} offer
three possible choices. We will see that the use of the local averages
in Lemma \ref{transConvergeLA} represents the best choice. Suppose
that on a component $\Pi^{(n)}(k)$ of the partition, we wish to
represent the $W^a$ with a constant $\theta_k$. The resulting mean
square error is given by
$$
\EE [(W^a-\theta_k)^2 \mid W^a \in \Pi^{(n)}(k)].
$$
It is well known that the choice of the conditional expectation
$\theta_k = \EE [W^a \mid W^a \in \Pi^{(n)}(k)]$ minimizes the above
mean square error. It is therefore clear that the disturbance sampling
in Lemma \ref{transConvergeLA} tends to represent the best
choice. However, the calculation of
$\EE [W^a \mid W^a \in \Pi^{(n)}(k)]$ may become impractical in some
settings e.g. high dimension state spaces and so the use of Lemma
\ref{transConvergeMC} or Lemma \ref{transConvergeBounded} may be the
only practical alternatives.

\subsection{Convergence of fixed points}

This subsection proves that fixed points from the modified value
iteration \eqref{modBell} converge to the fixed point of the original
value iteration \eqref{1bel} under the different sampling schemes
presented earlier. Let $(m_n)_{n\in\NN}$ and $(n_m)_{m\in\NN}$ be
sequences of natural numbers increasing in $n$ and $m$, respectively.

\begin{lemma}\label{lemmaRockafellar}
  Let $(h^{(n)})_{n\in\NN}$ be a sequence of real-valued convex
  functions on $\ZZZ$. If the sequence converges pointwise to $h$ on a
  dense subset of $\ZZZ$, then the sequence $(h^{(n)})_{n\in\NN}$
  converges uniformly to $h$ on all compact subsets of $\ZZZ$.
\end{lemma}
\begin{proof}
  See \cite[Theorem 10.8]{rockafellar}.
\end{proof}

\begin{lemma} \label{lemmaCCC} 
  Let $(h_1^{(n)})_{n\in\NN}$ and
  $(h_2^{(n)})_{n\in\NN}$ be CCC sequences on $\ZZZ$ converging to $h_1$
  and $h_2$, respectively. Define $h^{(n)}_3(z) := \max (h^{(n)}_1(z),
  h^{(n)}_2(z))$ and $h_3(z) := \max (h_1(z), h_2(z))$.  It holds that:
  \begin{itemize}
  \item $(h_1^{(n)} + h_2^{(n)})_{n\in\NN}$ is a CCC sequences on
    $\ZZZ$ converging to $h_1 + h_2$,
  \item $(\mathcal{S}_{\GGG^{m_n}}h_1^{(n)})_{n\in\NN}$ is a CCC
    sequences on $\ZZZ$ converging to $h_1$, and
  \item $(h_3^{(n)})_{n\in\NN}$ is a CCC sequences on $\ZZZ$ converging
    to $h_3$.
  \end{itemize}
\end{lemma}
\begin{proof}
  This follows from the definition of uniform convergence on compact
  sets, Assumption \ref{piecewiseS}, and Lemma \ref{lemmaRockafellar}.
\end{proof}

\begin{theorem} \label{valueConverge} Let $v^{(m,n),*}$ be the unique
  fixed point of $\mathcal{T}^{(m,n)}$. Using the disturbance
  samplings in either Lemma \ref{transConvergeMC}, Lemma
  \ref{transConvergeBounded} or Lemma \ref{transConvergeLA}, it holds
  that both $v^{(m_n,n),*}$ and $v^{(m,n_m),*}$ converge uniformly to
  $v^{\pi^*}$ on compact sets as $n\to\infty$ and $m\to\infty$,
  respectively.
\end{theorem}
\begin{proof} Let us consider the case of $n\to\infty$ first.  Let
  $v_0(p,z)\in\BBBB_b$ be convex in $z$ for all $p\in\PPP$ and denote
  $[\mathcal{T}^{(m_n,n)}]^i$ to represent $i\in\NN$ successive
  applications of the operator $\mathcal{T}^{(m_n,n)}$. At $i = 1$,
  $(v_0(p,z))_{n\in\NN}$ is a CCC seqeuence in $z$ and converges to
  $v_0(p,z)$ for $p\in\PPP$. Now, the use of Assumption
  \ref{assSampleConvex} combined with either Lemma
  \ref{transConvergeMC}, Lemma \ref{transConvergeBounded} or Lemma
  \ref{transConvergeLA} reveals that
  $\left(\mathcal{K}^{a,(n)}v_0(p,z)\right)_{n\in\NN}$ forms a CCC
  sequence in $z$ and that it converges to $\mathcal{K}^{a}v_0(p,z)$
  for all $p\in\PPP$ and $a\in\AAA$. Since $\AAA$ is finite, it holds
  from Lemma \ref{lemmaCCC} that
  $(\mathcal{T}^{(m_n,n)}v_0(p,z))_{n\in\NN}$ forms a CCC sequence in
  $z$ and converges to $\mathcal{T} v_0(p,z)$ for all $p\in\PPP$. Now
  at $i=2$, the same argument shows that
  $\left(\mathcal{K}^{a,(n)}\mathcal{T}^{(m_n,n)}v_0(p,z)\right)_{n\in\NN}$
  forms a CCC sequence in $z$ and that it converges to
  $\mathcal{K}^{a} \mathcal{T}v_0(p,z)$ for all $p\in\PPP$ and
  $a\in\AAA$. Again, using the same argument above also reveals that
  $([\mathcal{T}^{(m_n,n)}]^2v_0(p,z))_{n\in\NN}$ forms a CCC sequence
  in $z$ and converges to $[\mathcal{T}]^2 v_0(p,z)$ for all
  $p\in\PPP$. Proceeding inductively for $i=3,4,\dots$ proves that for
  all $i\in\NN$ that $[\mathcal{T}^{(m_n,n)}]^iv_0$ converges
  uniformly on compact sets to $[\mathcal{T}]^iv_0$. From the
  triangle inequality, it can be seen that
  \begin{multline*}
    \lim_{n\to\infty} \| v^{(m_n,n),*} - v^{\pi^*}\|_b \\ \leq
    \lim_{n\to\infty} \lim_{i\to\infty} \| v^{(m_n,n),*} -
    [\mathcal{T}^{(m_n,n)}]^iv_0 \|_b  +  \| [\mathcal{T}^{(m_n,n)}]^iv_0 -
    [\mathcal{T}]^iv_0 \|_b +   \|[\mathcal{T}]^iv_0 - v^{\pi^*}\|_b
  \end{multline*}
  and so from Lemma \ref{lemmaUC}, $v^{(m,n_m),*}$ converges to
  $v^{\pi^*}$ uniformly on compact sets. The proof for the
  $m\to\infty$ case follows the same lines as above.
\end{proof}

\section{Lower bounding fixed points} \label{secLower}

This section examines the conditions in which the fixed points from
\eqref{modBell} form a non-decreasing sequence of lower bounding
functions for the true value function. To this end, the following
assumption is required.  Assumption \ref{assMonotone} below induces
monotonicity in the modified Bellman operator \eqref{modBell} i.e. for
all $m,n\in\NN$, $p\in\PPP$, and $z\in\ZZZ$:
$\mathcal{T}^{(m,n)}v'(p,z) \leq \mathcal{T}^{(m,n)}v''(p,z)$ if
$v'(p',z) \leq v''(p',z)$ for all $p'\in\PPP$.

\begin{assumption} \label{assMonotone} Suppose for $z\in\ZZZ$ and
  $m\in\NN$ that
  ${\mathcal S}_{\GGG^{(m)}}h'(z) \leq {\mathcal
    S}_{\GGG^{(m)}}h''(z)$ for all convex functions
  $h',h'':\ZZZ \rightarrow \RR$ such that $h'(z) \leq h''(z)$ for
  $z\in\ZZZ$.
\end{assumption}

In the following, partition $\Pi^{(n+1)}$ is said to be a refinement
of $\Pi^{(n)}$ if each component of $\Pi^{(n+1)}$ is a subset of a
component in $\Pi^{(n)}$.

\begin{lemma} \label{lemmaK} Suppose $v(p,f(w,z))$ is convex in $w$
  for all $p\in\PPP$ and that $\Pi^{(n+1)}$ refines $\Pi^{(n)}$. Using
  the disturbance sampling in Theorem \ref{transConvergeLA} gives
  $ \mathcal{K}^{a,(n)}v(p,z) \leq \mathcal{K}^{a,(n+1)}v(p,z).  $
\end{lemma}
\begin{proof}
  See \cite[Lemma 4]{yee_convex}.
\end{proof}

\begin{theorem}\label{lowerBound}
  Using Theorem \ref{transConvergeLA} gives for $p\in\PPP$,
  $z\in\ZZZ$, and $m,n\in\NN$:
  \begin{itemize}
  \item $v^{(m,n),*}(p,z) \leq v^{\pi^*}(p,z)$ if $v^{\pi^*}(p', f(w, z'))$ is
    convex in $w$ and if
    ${\mathcal S}_{\GGG^{(m')}}h(z') \leq h(z')$ for $p'\in\PPP$,
    $z'\in\ZZZ$, and all convex functions
    $h$.
  \item $v^{(m,n),*}(p,z) \leq v^{(m,n+1),*}(p,z) $ if $\Pi^{(n'+1)}$
    refines $\Pi^{(n')}$ and if $v^{(m',n'),*}(p', f(w, z'))$ is
    convex in $w$ for $p'\in\PPP$, $z'\in\ZZZ$, and $m',n'\in\NN$.
  \item $v^{(m,n),*}(p,z) \leq v^{(m+1,n),*}(p,z)$ if
    ${\mathcal S}_{\GGG^{(m')}}h(z') \leq {\mathcal
      S}_{\GGG^{(m'+1)}}h(z')$ for $z'\in\ZZZ$, $m'\in\NN$, and all
    convex functions $h$.
  \end{itemize}
\end{theorem}
\begin{proof} We prove the three inequalities separately.

  1) Recall that $W^{a,(n)} = \EE[W^{a} \mid \sigma^{(n)}_{a}]$. The
  tower property and Jensen's inequality gives for $a\in\AAA$,
  $p\in\PPP$, $z\in\ZZZ$, and $n\in\NN$:
  \begin{align*}
    \EE[v^{\pi^*}(p, f(W^{a},z))]
    \geq \EE[v^{\pi^*}(p, f(W^{a,(n)},z))] = \sum_{k=1}^n \rho^{a,(n)}(k) v^{\pi^*}(p, f(W^{a,(n)}(k), z))
  \end{align*}
  $\implies \mathcal{K}^{a,(n)}_{T-1} v^{\pi^*}(p,z) \leq
  \mathcal{K}^{a}_{T-1}v^{\pi^*}(p,z)$
  $\implies \mathcal{T}^{(m,n)} v^{\pi^*}(p,z) \leq \mathcal{T}
  v^{\pi^*}(p,z) = v^{\pi^*}(p,z)$ for $p\in\PPP$ and $z\in\ZZZ$ by
  Assumption \ref{assMonotone} and
  ${\mathcal S}_{\GGG^{(m)}}r(p,z,a) \leq r(p,z,a)$. Now using the
  above and the monotonicity of the modified Bellman operator
  \eqref{modBell} gives
  $$
  [\mathcal{T}^{(m,n)}]^2 v^{\pi^*}(p,z) \leq
  \mathcal{T}^{(m,n)}\mathcal{T} v^{\pi^*}(p,z) \leq \mathcal{T}^{(m,n)}
  v^{\pi^*}(p,z) \leq \mathcal{T} v^{\pi^*}(p,z) = v^{\pi^*}(p,z).
  $$
  In the same manner, it can be shown that for $i\in\NN$:
  $$
  [\mathcal{T}^{(m,n)}]^i v^{\pi^*}(p,z) \leq [\mathcal{T}]^i
  v^{\pi^*}(p,z) = v^{\pi^*}(p,z) 
  $$
  and so $v^{(m,n),*}(p,z) \leq v^{\pi^*}(p,z)$ for all $p\in\PPP$,
  $z\in\ZZZ$, and $m,n\in\NN$.

  2) Lemma \ref{lemmaK} implies
  $\mathcal{K}^{a, (n)} v^{(m,n)*}_T(p,z) \leq \mathcal{K}^{a,
    (n+1)}v^{(m,n),*}_T(p,z)$. From Assumption
  \ref{assMonotone}, it can be shown that
  $\mathcal{T}^{(m,n)} v^{(m,n),*}(p,z) \leq \mathcal{T}^{(m,n+1)}
  v^{(m,n),*}(p,z)$. Using this,
  $$
  v^{(m,n),*}(p,z) \leq \mathcal{T}^{(m,n+1)}\mathcal{T}^{(m.n)}
  v^{(m,n),*}(p,z) \leq [\mathcal{T}^{(m,n+1)}]^2 v^{(m.n),*}(p,z)
  $$
  by the monotonicity of the modified Bellman operator.  Therefore, it
  can be shown for $i\in\NN$:
  $$
  v^{(m,n),*}(p,z) \leq [\mathcal{T}^{(m,n+1)}]^i v^{(m.n),*}(p,z)
  $$
  and so $v^{(m,n),*}(p,z) \leq v^{(m,n+1),*}(p,z)$ for all
  $p\in\PPP$, $z\in\ZZZ$, and $m,n\in\NN$.

  3) By assumption
  ${\mathcal S}_{\GGG^{(m)}}r(p,z,a) \leq {\mathcal
    S}_{\GGG^{(m+1)}}r(p,z,a)$ for $p\in\PPP$, $a\in\AAA$ and
  $m\in\NN$. Therefore,
  $\mathcal{T}^{(m,n)} v^{(m,n),*}(p,z) \leq \mathcal{T}^{(m+1,n)}
  v^{(m,n),*}(p,z)$. Using a similar argument as before, it can be shown
  that for $i\in\NN$:
  $$
  v^{(m,n),*}(p,z) \leq [\mathcal{T}^{(m+1,n)}]^i v^{(m.n),*}(p,z)
  $$
  and so $v^{(m,n),*}(p,z) \leq v^{(m+1,n),*}(p,z)$ for all
  $p\in\PPP$, $z\in\ZZZ$, and $m,n\in\NN$.

\end{proof}

\section{Upper bounding fixed points} \label{secUpper}

This section examines conditions in which the fixed points from the
modified value iteration \eqref{modBell} form a non-increasing
sequence of upper bounding functions for the true value
function. Suppose $\WWW$ is compact.  Assume that the closures of each
of the components in partition $\Pi^{(n)}$ are convex and contain a
finite number of extreme points. Denote the set of extreme points of
the closure $\widetilde\Pi^{(n)}(k)$ of each component $\Pi^{(n)}(k)$
in partition $\Pi^{(n)}$ by
$$
\mathcal{E}(\widetilde\Pi^{(n)}(k)) = \{e^{(n)}_{k,i} : i = 1, \dots, L_k^{(n)}\}
$$
where $L_k^{(n)}$ is the number of extreme points in
$\widetilde\Pi^{(n)}(k)$. Suppose there exist weighting functions
$q^{a,(n)}_{k,i}:\WWW\to[0,1]$ satisfying
\begin{equation*}
  \sum_{i=1}^{L^{(n)}_k}q^{a,(n)}_{k,i}(w) = 1 
  \quad \text{and} \quad
  \sum_{i=1}^{L^{(n)}_k}q^{a,(n)}_{k,i}(w) e^{(n)}_{k,i} = w 
\end{equation*}
for all $w\in\widetilde\Pi^{(n)}(k)$ and $k=1,\dots,n$. Suppose
$\rho^{a,(n)}(k) = \PP(W^{a} \in \Pi^{(n)}(k)) > 0$ for
$k=1,\dots,n$ and define random variables
$\overline{W}_{k}^{a,(n)}$ satisfying
$$
\PP\left(\overline{W}_{k}^{a,(n)}=e^{(n)}_{k,i}\right) =
\frac{\bar q^{a,(n)}_{k,i}}{\rho^{a,(n)}_{t+1}(k)} \quad
\text{where} \quad \bar q^{a,(n)}_{k,i} =
\int_{\widetilde\Pi^{(n)}(k)} q^{a,(n)}_{k,i}(w)
\mu^{a}(\mathrm{d}w)
$$
and $\mu^{a}(\mathbf{B}) = \PP(W^{a} \in \mathbf{B})$.
For the next theorem, define random variable
$$
\overline{W}^{a,(n)} = \sum_{k=1}^{n} \overline{W}_{k}^{a,(n)} \11\left(W^{a}\in\Pi^{(n)}(k)\right)
$$
and introduce the following alternative modified transition operator:
\begin{equation} \label{upperModKern}
{\mathcal K}^{a, (n)}v(p, z) = \sum_{p'\in\PPP}
\alpha^a_{p,p'}\EE[v(p',f(\overline{W}^{a,(n)},z))].
\end{equation}

The following theorem proves convergence of the fixed points when the
above is used in the modified value iteration.

\begin{lemma} \label{transConvergeUpper} If
  $\lim_{n\to\infty} \delta^{(n)} = 0$, then
  $$
  \lim_{n\to\infty} \EE[v^{(n)}(p, f(\overline{W}^{a,(n)},z))] =
  \EE[\widehat v (p, f(W^{a},z))], \quad z\in\ZZZ.
  $$
  If $\EE[v^{(n)}(p, f(\overline{W}^{a,(n)},z))]$ is also convex in $z$ for
  $n\in\NN$, then
  $(\EE[v^{(n)}(p, f(\overline{W}^{a,(n)},z))])_{n\in\NN}$ form a CCC
  sequence in $z$.
\end{lemma}
\begin{proof} 
  See \cite[Lemma 6]{yee_convex}.
 \end{proof}
  
\begin{theorem} \label{valueConvergeUpper} Suppose
  $\lim_{n\to\infty} \delta^{(n)} = 0$.  Let $v^{(m,n),*}$ be the
  unique fixed point of $\mathcal{T}^{(m,n)}$ where the modified
  transition operator \eqref{upperModKern} is used. It holds that both
  $v^{(m_n,n),*}$ and $v^{(m,n_m),*}$ converge uniformly to
  $v^{\pi^*}$ on compact sets as $n\to\infty$ and $m\to\infty$,
  respectively.
\end{theorem}
\begin{proof} The proof mirrors the proof in Theorem
  \ref{valueConverge} but using Lemma \ref{transConvergeUpper}
  instead.
\end{proof}

To show upper bounding behaviour of the modified fixed points, the
following lemma is handy.

\begin{lemma} \label{transUpper} Suppose $\Pi^{(n+1)}$ refines
  $\Pi^{(n)}$ and that $v^{(n)}(p, f(w, z))$ is convex in $w$ for
  $p\in\PPP$, $z\in\ZZZ$, and $n\in\NN$. Assume
  $v^{(n)}(p,z) \geq v^{(n+1)}(p,z) \geq \widehat v(p,z) $ for
  $p\in\PPP$, $z\in\ZZZ$, and $n\in\NN$. Then
  $$
  \EE[v^{(n)}(p,f(\overline{W}^{a,(n)},z))] \geq
  \EE[v^{(n+1)}(p, f(\overline{W}^{a,(n+1)},z))]
  \geq \EE[\widehat v (p, f({W}^{a},z))]
  $$
  for $a\in\AAA$, $p\in\PPP$, $z\in\ZZZ$, and $n\in\NN$. 
\end{lemma}
\begin{proof} See \cite[Theorem 7.8]{hernandezlerma_runggaldier1994}
  or \cite[Theorem 1.3]{hernandezlerma_etal1995}.
 \end{proof}
 
\begin{theorem}\label{upperBound}
  Using the modified transition operator \eqref{upperModKern} gives for
  $p\in\PPP$, $z\in\ZZZ$, and $m,n\in\NN$:
  \begin{itemize}
  \item $v^{(m,n),*}(p,z) \geq v^{\pi^*}(p,z)$ if $v^{\pi^*}(p', f(w, z'))$ is
    convex in $w$ and if
    ${\mathcal S}_{\GGG^{(m')}}h(z') \geq h(z')$ for $p'\in\PPP$,
    $z'\in\ZZZ$, and all convex functions
    $h$.
  \item $v^{(m,n),*}(x) \geq v^{(m,n+1),*}(x) $ if $\Pi^{(n'+1)}$ is a
    refinement of $\Pi^{(n')}$ and if $v^{(m',n'),*}(p', f(w, z'))$ is
    convex in $w$ for $p'\in\PPP$, $z'\in\ZZZ$, and $m',n'\in\NN$.
  \item $v^{(m,n),*}(x) \leq v^{(m+1,n),*}(x)$ if
    ${\mathcal S}_{\GGG^{(m')}}h(z') \geq {\mathcal
      S}_{\GGG^{(m'+1)}}h(z')$ for $z'\in\ZZZ$, $m'\in\NN$, and all
    convex functions $h$.
  \end{itemize}
\end{theorem}
\begin{proof} The inequalities are basically reversed in the proof of
  Theorem \ref{lowerBound}:

  1) Lemma \ref{transUpper}
  $\implies \mathcal{K}^{a,(n)} v^{\pi^*}(p,z) \geq
  \mathcal{K}^{a}v^{\pi^*}(p,z)$
  $\implies \mathcal{T}^{(m,n)} v^{\pi^*}(p,z) \geq \mathcal{T}
  v^{\pi^*}(p,z) = v^{\pi^*}(p,z)$ for $p\in\PPP$ and $z\in\ZZZ$ by
  Assumption \ref{assMonotone} and the fact that
  ${\mathcal S}_{\GGG^{(m)}}r(p,z,a) \geq r(p,z,a)$. Now using the
  above and the monotonicity of the modified Bellman operator
  \eqref{modBell} gives
  $$
  [\mathcal{T}^{(m,n)}]^2 v^{\pi^*}(p,z) \geq
  \mathcal{T}^{(m,n)}\mathcal{T} v^{\pi^*}(p,z) \geq \mathcal{T}^{(m,n)}
  v^{\pi^*}(p,z) \geq \mathcal{T} v^{\pi^*}(p,z) = v^{\pi^*}(p,z).
  $$
  In the same manner, it can be shown that for $i\in\NN$:
  $$
  [\mathcal{T}^{(m,n)}]^i v^{\pi^*}(p,z) \geq [\mathcal{T}]^i
  v^{\pi^*}(p,z) = v^{\pi^*}(p,z) 
  $$
  and so $v^{(m,n),*}(p,z) \leq v^{\pi^*}(p,z)$ for all $p\in\PPP$,
  $z\in\ZZZ$, and $m,n\in\NN$.

  2) Lemma \ref{transUpper}
  $\implies \mathcal{K}^{a, (n)} v^{(m,n)*}_T(p,z) \geq
  \mathcal{K}^{a, (n+1)}v^{(m,n),*}_T(p,z)$. From Assumption
  \ref{assMonotone}, it can be shown that
  $\mathcal{T}^{(m,n)} v^{(m,n),*}(p,z) \geq \mathcal{T}^{(m,n+1)}
  v^{(m,n),*}(p,z)$. Using this,
  $$
  v^{(m,n),*}(p,z) \geq \mathcal{T}^{(m,n+1)}\mathcal{T}^{(m.n)}
  v^{(m,n),*}(p,z) \geq [\mathcal{T}^{(m,n+1)}]^2 v^{(m.n),*}(p,z)
  $$
  by the monotonicity of the modified Bellman operator.  Therefore, it
  can be shown for $i\in\NN$:
  $$
  v^{(m,n),*}(p,z) \geq [\mathcal{T}^{(m,n+1)}]^i v^{(m.n),*}(p,z)
  $$
  and so $v^{(m,n),*}(p,z) \geq v^{(m,n+1),*}(p,z)$ for all
  $p\in\PPP$, $z\in\ZZZ$, and $m,n\in\NN$.

  3) By assumption
  ${\mathcal S}_{\GGG^{(m)}}r(p,z,a) \geq {\mathcal
    S}_{\GGG^{(m+1)}}r(p,z,a)$ for $p\in\PPP$, $a\in\AAA$ and
  $m\in\NN$. Therefore,
  $\mathcal{T}^{(m,n)} v^{(m,n),*}(p,z) \geq \mathcal{T}^{(m+1,n)}
  v^{(m,n),*}(p,z)$. Using a similar argument as for the first two
  cases, it can be shown that for all $i\in\NN$:
  $$
  v^{(m,n),*}(p,z) \geq [\mathcal{T}^{(m+1,n)}]^i v^{(m.n),*}(p,z)
  $$
  and so $v^{(m,n),*}(p,z) \geq v^{(m+1,n),*}(p,z)$ for all
  $p\in\PPP$, $z\in\ZZZ$, and $m,n\in\NN$.
\end{proof}

\section{Numerical study}
\label{sec_numerical}

As way of demonstration, a perpetual Bermudan put option in the
Black-Scholes world is considered.  A Bermudan put option represents
the right but not the obligation to sell the underlying asset for a
predetermined strike price $K$ at discrete time points.  This problem
is given by $\mathbf{P}=\{\text{exercised}, \text{unexercised}\}$ and
$\mathbf{A}=\{\text{exercise}, \text{don't exercise}\}$. At $P_t=$
``unexercised'', applying $a=$ ``exercise'' and $a=$ ``don't
exercise'' leads to $P_{t+1}=$ ``exercised'' and $P_{t+1}=$
``unexercised'', respectively with probability one. If $P_t=$
``exercised'', then $P_{t+1}$ = ``exercised'' almost surely regardless
of any action. Let $\kappa > 0$ represent the interest rate per annum
and let $z$ represent the underlying stock price. Defining
$(z)^+ = \max(z,0)$, the reward/payoff function is given by
$$
r(\text{unexercised}, z, \text{exercise}) = (K -  z)^+
$$
for $z \in \RR_{+}$ and zero for other $p\in\PPP$ and
$a\in\AAA$. The fair price of the option is
\begin{align*}
  & v^{\pi^*}(\text{unexercised}, z_0) =\EE^{(\text{unexercised},z_0),
    \pi^*}\left(\sum_{t=0}^{\infty} \beta^t r_{t}(X_{t}, \pi^*_{t}(X_{t}))
    \right) \quad \text{or}
  \\ 
  & v^{\pi^*}(\text{unexercised}, z_0) = \sup\left\{\EE[\beta^{\tau}(K- Z_{\tau})^+]: \tau \in\NN\right\}
\end{align*}
where $\beta = e^{-\kappa \Delta}$ is the discount factor and $\Delta$
is the time step. Assume
\begin{equation*}
  Z_{t+1}=  W_{t+1} Z_t = e^{(\kappa - \frac{\text{vol}^2}{2})\Delta + \text{vol}\sqrt{\Delta}N_{t+1}} Z_{t}
\end{equation*}
where $(N_{t})_{t\in\NN}$ are independent and identically distributed
standard normal random variables and $\text{vol}$ is the volatility of
stock price returns. 

Note that the disturbance is not controlled by action $a$ and so the
superscript is removed from $W^a_{t+1}$ for notational simplicity in
the following subsections.  It is clear that
$$
|r(p,z,a)| \leq b(p,z) = K, \quad \EE[K] \leq K, \quad \text{and}
\quad \beta < 1
$$
for $p\in\PPP$, $z\in\ZZZ$, and $a\in\AAA$. So by Assumption
\ref{assContract}, the perpetual Bermundan put option is a contracting
Markov decision process with a unique fixed point $v^{\pi^*}$ that can
be found via value iteration. In the next two subsections, different
modifed value iteration schemes are presented. The first results in
the modified fixed points forming a lower bound and the second forms
an upper bound. The same method were employed in a finite horizon with
great success (see \cite[Section 7]{yee_convex}) and this also holds
true in our infinite horizon setting.

\subsection{Lower bound via tangents} \label{secTangents}

The following scheme approximates the convex functions using the
maximum of their tangents. From \cite[Theorem 25.5]{rockafellar}, it
is known that a convex real valued function is differentiable almost
everywhere. Suppose convex function $h:\ZZZ \rightarrow \RR$ holds
tangents on each point in $\GGG^{(m)}$ given by
$\{h_1'(z), \dots, h_m'(z)\}$. Set the approximation scheme
${\mathcal S}_{\GGG^{(m)}}$ to take the maximising tangent to form a
convex piecewise linear approximation of $h$ i.e.
$$
{\mathcal S}_{\GGG^{(m)}}h(z) = \max\{h_{1}'(z),\dots, h_m'(z)\}.
$$  
It is not hard to see that the resulting approximation
${\mathcal S}_{\GGG^{(m)}}h$ is convex, piecewise linear, and
converges to $h$ uniformly on compact sets as $m\to\infty$.  Thus,
Assumption \ref{piecewiseS} is satisfied. Now observe that Assumption
\ref{assContinuity} and \ref{assSampleConvex} both holds since
$(Z_t)_{t\in\NN}$ evolves in a linear manner. By definition of
${\mathcal S}_{\GGG^{(m)}}$, Assumption \ref{assMonotone} is
obeyed. Let us now verify Assumption \ref{assModContract}. Note that
${\mathcal S}_{\GGG^{(m)}}h(z) \leq {\mathcal S}_{\GGG^{(m+1)}}h(z)
\leq h(z)$ for all $z\in\ZZZ$.  Observe that the reward and true value
functions are bounded above by $K$ and below by $0$. Therefore, if the
approximation scheme ${\mathcal S}_{\GGG^{(m)}}$ includes a zero
tangent i.e $h'_m(z) = 0$ for all $z\in\ZZZ$, then all the conditions
in Assumption \ref{assModContract} are true. Therefore, the modified
Bellman operator constructed using the above function approximation
gives a contracting Markov decision process. 

For the distubance sampling, the space $\WWW = \RR_+$ is partitioned
into sets of equal probability measure and then the local averages on
each component are used.  It is not difficult to verify that all the
conditions in Theorem \ref{lowerBound} hold and so the fixed point of
this modified value iteration forms a non-decreasing sequence of lower
bounding functions for the true value function. Before proceeding to
the next scheme, let us address the issue of convergence in the above
modified value iteration. Observe in \eqref{modBell} that if the
expected value function converges to its fixed point, then the whole
value iteration does so as well. Since $\ZZZ$ is uncountable, it is
impossible to verify this convergence directly for each
$z\in\ZZZ$. However, note that if
$$
\S_{\GGG^{(m)}}h_1(z) = \S_{\GGG^{(m)}}h_2(z) \quad \text{and} \quad
\nabla_z\S_{\GGG^{(m)}}h_1 = \nabla_z \S_{\GGG^{(m)}}h_2
$$
for all $z\in\GGG^{(m)}$ where the $\nabla_z$ operator gives the
gradient/slope at $z$, then
$\S_{\GGG^{(m)}}h_1(z) = \S_{\GGG^{(m)}}h_2(z)$ for all
$z\in\ZZZ$. Therefore, there is convergence in the value iteration
when there is convergence of the tangents on each of the finite number
of grid points. In this manner, this approach is numerically
tractable.

\subsection{Upper bound via linear interpolation}

The following constructs the upper bounds using Section
\ref{secUpper}. Let us approximate unbounded $\WWW = \RR_+$ with a
compact set $\overline{\WWW}$ containing $99.9999999\%$ of the
probability mass i.e.  $\PP(W_t \in \overline{\WWW}) = 0.999999999$
for all $t\in\NN$. Introduce the truncated distribution $\overline\PP$
defined by
$\overline\PP(W_t\in\mathbf{B}) = \alpha \PP(W_t\in\mathbf{B})$ for
all $\mathbf{B}\subseteq\overline\WWW$ where
$\alpha = 1/\PP(W_t \in \overline{\WWW})$ is the normalizing constant.
Suppose the partition $\Pi^{(n)}$ comprises of $n$ convex components
of equal probability measure and that the extreme points are ordered
$e^{(n)}_{k,1} < e^{(n)}_{k,2}$ for all $k=1,\dots,n$. For
$k=1,\dots, n$, define points
$e^{(n)}_k = e^{(n)}_{\floor{(k+1)/2}, \floor{k+2-2\floor{(k+1)/2}}}$
and $e^{(n)}_{n+1} = e^{(n)}_{n,2}$ where $\floor{\hspace{1mm}}$
denotes the integer part. Define
$\Lambda(a,b) := \EE^{\overline\PP}[W_1 \11(W_1 \in[a,b])]$ and set
$$
q^{(n)}_{k,1}(w) = \frac{e_{j,2} - w}{e_{j,2}-e_{j,1}} 
\quad \text{and} \quad
q^{(n)}_{k,2}(w) = \frac{w - e_{j,1}}{e_{j,2}-e_{j,1}}. 
$$
Therefore,
$$
\overline\PP\left(\overline{W}_{1}^{(n)} = e_1^{(n)}\right)
= \frac{\left(e^{(n)}_2/n - \Lambda(e^{(n)}_1,e^{(n)}_2) \right)}
{e^{(n)}_2-e^{(n)}_1}, 
$$
$$
\overline\PP\left(\overline{W}_{1}^{(n)} = e_{n+1}^{(n)}\right)
= \frac{\left(\Lambda (e^{(n)}_{n},e^{(n)}_{n+1}) - e^{(n)}_{n}/n \right)}
{e^{(n)}_{n+1}-e^{(n)}_n}, 
$$
$$
\overline\PP\left(\overline{W}_{1}^{(n)} = e_{j}^{(n)}\right) =
\frac{\left(e^{(n)}_{j+1}/n - \Lambda (e^{(n)}_{j},e^{(n)}_{j+1})
  \right)} {e^{(n)}_{j+1}-e^{(n)}_j} + \frac{\left(
    \Lambda (e^{(n)}_{j-1},e^{(n)}_{j}) - e^{(n)}_{j-1}/n \right)}
{e^{(n)}_{j}-e^{(n)}_{j-1}}
$$
for $j=2,\dots,n$.

It is well known that
$v^{\pi^*}(\text{unexercised},z) =
r(\text{unexercised},z,\text{exercise})$ when $z < z'$ for some
$z'\in\ZZZ$ and that the value function is decreasing in $z$. Thease
features will be exploited. Suppose $h:\ZZZ\to\RR$ is a decreasing
convex function and $h(z)=h'(z)$ when $z<z'$.  For
$\GGG^m = \{g^{(1)},\dots, g^{(m)}\}$ where $g^{(1)} \leq z'$ and
$g^{(1)}<g^{(2)}<\dots<g^{(m)}$, set
$$
\S_{\GGG^{(m)}}h(z) =
 \Bigg\{ \begin{array}{ll}
        h'(z) & \mbox{if $z \leq g^{(1)}$};\\
        d_i (z - g^{(i)}) +   h(g^{(i)}) & \mbox{if $g^{(i)} < z \leq g^{(i+1)}$};\\
        h(g^{(m)}) & \mbox{if $z > g^{(m)}$},\end{array} 
$$
where $d_i =\frac{h(g^{(i+1)}) - h(g^{(i)})}{g^{(i+1)} - g^{(i)}} $
for $i=2,\dots,m-1$. It is not difficult to verify that under this
function approximation scheme, all the relevant assumptions are
satisifed and so the above gives a contracting Markov decision
process. Further, all the conditions in Theorem \ref{upperBound} is
also satisfied and so the fixed points from this modified value
iteration form a non-increasing sequence of upper bounding functions
for the true value function. Finally, let us mention about the
stopping criterion for the above modified criterion. It can be shown
that if 
$$
\S_{\GGG^{(m)}}h_1(z) = \S_{\GGG^{(m)}}h_2(z)  
$$
for all $z\in\GGG^{(m)}$, then
$\S_{\GGG^{(m)}}h_1(z) = \S_{\GGG^{(m)}}h_2(z)$ for all
$z\in\ZZZ$. Thus, there is convergence in the value iteration when the
function values on each grid point converges.
 
\subsection{Numerical results}

The following results were generated on a Bermudan put option with
strike price $40$. The put option is assumed to be exercisable every 3
months for perpetuity until exercised. The interest rate is set at
$0.15$ per annum. The reward function was used as the seeding function
in the value iteration. If the values or gradients at each grid point
is within $0.001$ of the last iteration, the value iteration is
stopped and we assume convergence. The computational times listed are
for Linux Ubuntu 16.04 machine with Intel i5-5300U CPU @2.30GHz and
16GB of RAM. The following results can be reproduced using the
\emph{R} script listed in the appendix. For the disturbace sampling, a
partition of $n=1000$ components of equal probability measure is used.

\begin{table}[h]
  \caption{Option valuation with different volatilities.
    \label{tablePut}}
  \centering
 \setlength\tabcolsep{3.5pt} 
  \begin{tabular}{l|ccc|ccc|ccc}
    &\multicolumn{3}{c}{vol = $0.1$ p.a.}&\multicolumn{3}{c}{vol = $0.2$ p.a.} &\multicolumn{3}{c}{vol = $0.3$ p.a.}\\
    $Z_0$ & Lower & Upper & Gap & Lower & Upper & Gap & Lower & Upper & Gap\\
    \hline
    32 &  8.00000 &  8.00000 & 0.00000 &  8.00000 &  8.00000 & 0.00000 &  8.00000 &  8.00000 & 0.00000 \\
    34 &  6.00000 &  6.00000 & 0.00000 &  6.00000 &  6.00000 & 0.00000 &  6.28550 &  6.30199 & 0.01649 \\
    36 &  4.00000 &  4.00000 & 0.00000 &  4.00000 &  4.00000 & 0.00000 &  5.23546 &  5.25366 & 0.01820 \\
    38 &  2.00000 &  2.00000 & 0.00000 &  2.45520 &  2.47724 & 0.02204 &  4.38277 &  4.40150 & 0.01874 \\
    40 &  0.34539 &  0.37316 & 0.02776 &  1.69317 &  1.71520 & 0.02203 &  3.69464 &  3.71292 & 0.01828 \\
    42 &  0.08485 &  0.09846 & 0.01361 &  1.17535 &  1.19501 & 0.01966 &  3.13829 &  3.15556 & 0.01727 \\
    44 &  0.02030 &  0.02556 & 0.00526 &  0.82723 &  0.84366 & 0.01643 &  2.68569 &  2.70162 & 0.01593 \\
    46 &  0.00508 &  0.00745 & 0.00237 &  0.59119 &  0.60451 & 0.01332 &  2.31435 &  2.32890 & 0.01455 \\
  \end{tabular}
\end{table}

Recall the the approximation schemes give lower and upper bounding
functions for the fair price of the perpetual option.  Table
\ref{tablePut} list points on these curves at different starting asset
prices. Columns $2$ to $4$ give an option with $vol=0.1$ per
annum. For this case, a grid of $51$ equally spaced points from $z=20$
to $z=70$ was used. The computational times for each bounding function
is around $0.025$ cpu seconds ($0.005$ real world seconds) and it took
roughly $10$ iteration for both schemes to converge.  Columns $5$ to
$7$ gives an option with $vol = 0.2$.  Here, a larger grid of $101$
equally spaced points from $z=20$ to $z=120$ was used. The
computational times for each bounding function is around $0.05$ cpu
seconds ($0.01$ to $0.03$ real world seconds). Convergence took $45$
and $28$ iterations for the lower and upper bounding schemes,
respectively. Columns $8$ to $10$ gives the case $vol = 0.3$.  A grid
of $401$ equally spaced points from $z=20$ to $z=420$ was used. The
computational times for each bounding function is around $0.2$ cpu
seconds ($0.05$ to $0.1$ real world seconds). Convergence took $69$
and $52$ iterations for the lower and upper bounding schemes,
respectively. The lower bounding and upper bounding curves are plotted
in Figure \ref{plotComparePut2} below.

\begin{figure}[h!]
  \centering
  \includegraphics[height=2in,width=0.45\textwidth]{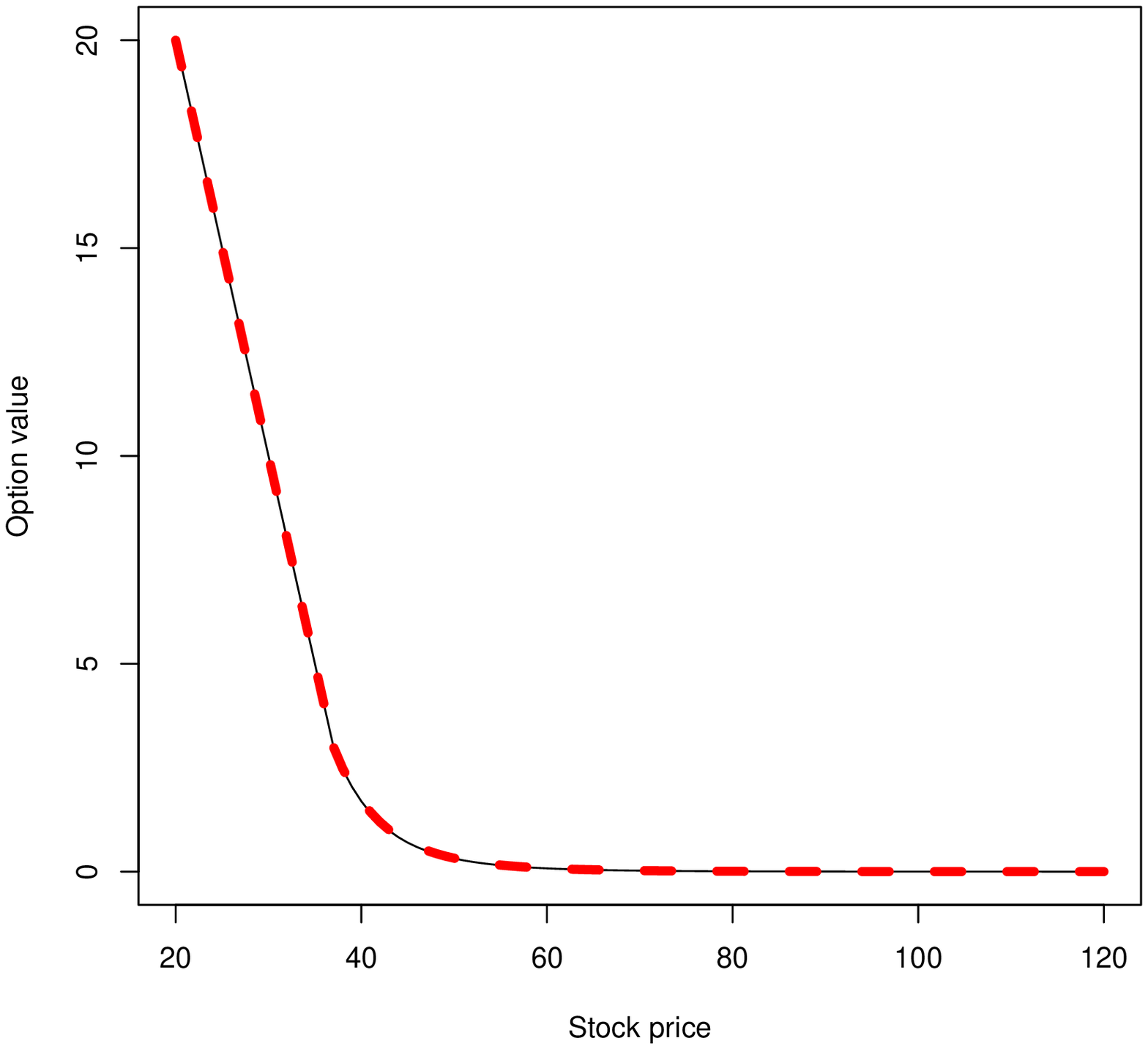}
  \includegraphics[height=2in,width=0.45\textwidth]{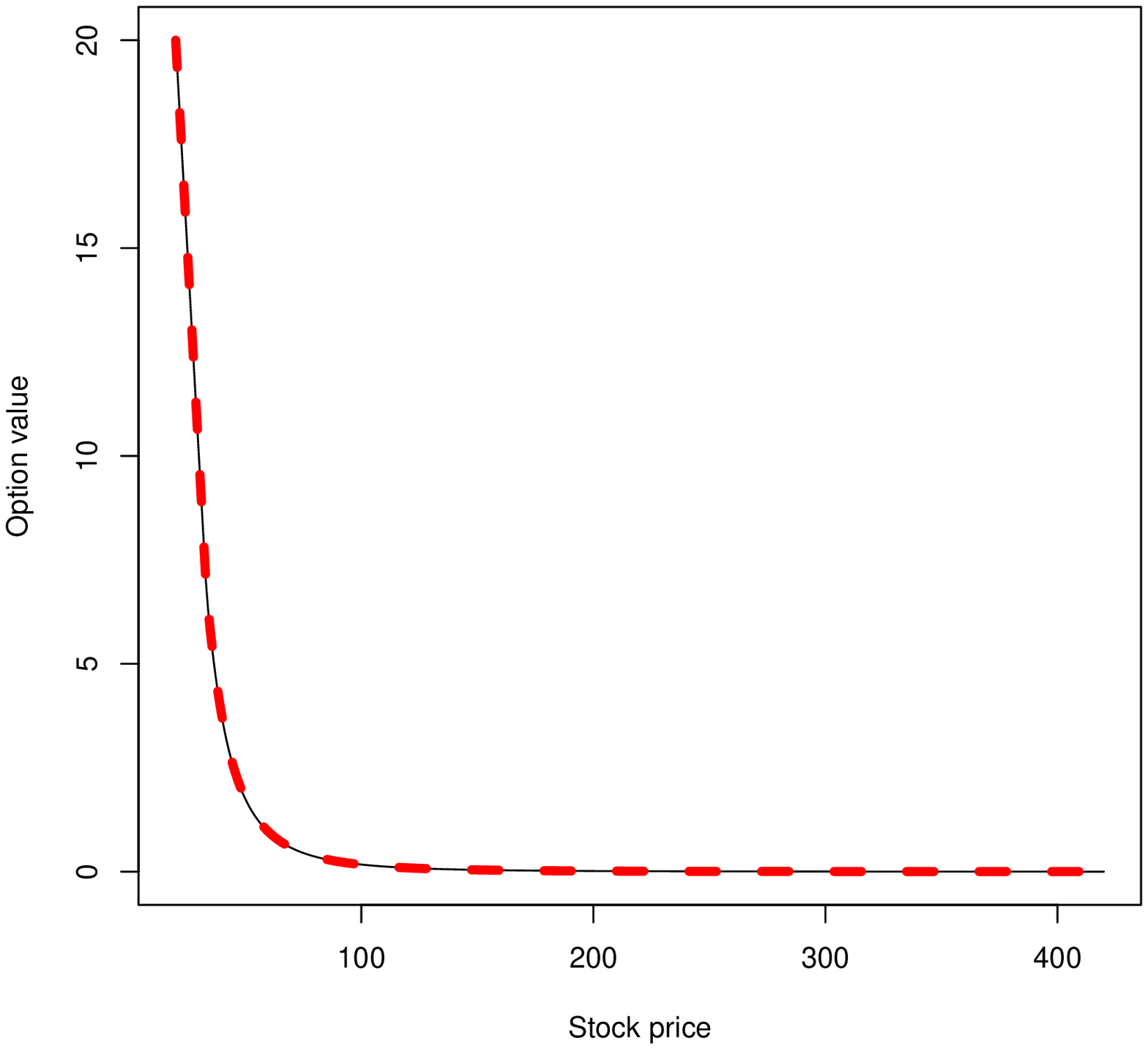}
  \caption{Lower and upper bounding functions for option price with
    $vol=0.2$ (left plot) and $vol=0.3$ (right plot). The dashed lines
    indicate the upper bound while the unbroken curves give the lower
    bounds.}
  \label{plotComparePut2}
\end{figure}

\section{Final thoughts}
\label{sec_conclusion}

In this paper, the original value iteraion is replaced with a more
tractable approximation \eqref{modBell}. Under certain convexity
assumptions, the fixed points from this modified value iteration
converges uniformly on compact sets to the true value functions under
different sampling schemes for the driving random
disturbance. Moreover, the fixed points from the approximate value
iteration \eqref{modBell} form a monotone sequence of functions that
bound the true value function. The results in this paper can be
modified for problems involving concave functions. For example,
problems of the form
$$
{\mathcal T}v(p,z)=\min_{a \in \AAA} (r(p,z,a)+ {\mathcal
  K}^{a}v(p,z))
$$
where the reward function is concave in $z$ and the transition
operator preserves concavity. Extensions to partially observable
Markov decision processes will be considered in future research.

\appendix
\section{R Script for Table \ref{tablePut}}

The script (along with the R package) used to generate columns 2, 3,
and 4 in Table \ref{tablePut} can be found at
\url{https://github.com/YeeJeremy/ConvexPaper}.  To generate the
others, simply modify the values on Line 5, Line 11, and Line 12.

\bibliography{main}

\providecommand{\bysame}{\leavevmode\hbox to3em{\hrulefill}\thinspace}
\providecommand{\MR}{\relax\ifhmode\unskip\space\fi MR }
\providecommand{\MRhref}[2]{%
  \href{http://www.ams.org/mathscinet-getitem?mr=#1}{#2}
}
\providecommand{\href}[2]{#2}
\begin{thebibliography}{10}

\bibitem{bauerle_rieder}
N.~Bauerle and U.~Rieder, \emph{Markov decision processes with applications to
  finance}, Springer, Heidelberg, 2011.

\bibitem{blackwell1965}
D.~Blackwell, \emph{Discounted dynamic programming}, Ann. Math. Statist.
  \textbf{36} (1965), no.~1, 226--235.

\bibitem{denardo1967}
E.~Denardo, \emph{Contraction mappings in the theory underlying dynamic
  programming}, SIAM Review \textbf{9} (1967), no.~2, 165--177.

\bibitem{hernandezlerma_lasserreBook}
O.~Hernandez-Lerma and J.~Lasserre, \emph{Discrete-time markov control
  processes : Basic optimality criteria}, Springer, New York, 1996.

\bibitem{hernandezlerma_etal1995}
O.~Hernandez-Lerma, C.~Piovesan, and W.~Runggaldier, \emph{Numerical aspects of
  monotone approximations in convex stochastic control problems}, Annals of
  Operations Research \textbf{56} (1995), no.~1, 135--156.

\bibitem{hernandezlerma_runggaldier1994}
O.~Hernandez{-}Lerma and W.~Runggaldier, \emph{Monotone approximations for
  convex stochastic control problems}, Journal of Mathematical Systems,
  Estimation, and Control \textbf{4} (1994), no.~1, 99--140.

\bibitem{hinderer}
K.~Hinderer, \emph{Foundations of non-stationary dynamic programming with
  discrete time parameter}, Springer-Verlag, Berlin, 1970.

\bibitem{howard}
R~Howard, \emph{Dynamic programming and markov processes}, MIT Press, 1960.

\bibitem{powell}
W.~Powell, \emph{Approximate dynamic programming: Solving the curses of
  dimensionality}, Wiley, Hoboken, New Jersey, 2007.

\bibitem{puterman}
M.~Puterman, \emph{Markov decision processes: Discrete stochastic dynamic
  programming}, Wiley, New York, 1994.

\bibitem{rockafellar}
R.~Rockafellar, \emph{Convex analysis}, Princeton landmarks in mathematics and
  physics, Princeton University Press, 1970.

\bibitem{shapley1953}
L.~Shapley, \emph{Stochastic games}, Proceedings of the National Academy of
  Sciences \textbf{39} (1953), no.~10, 1095--1100.

\bibitem{yee_convex}
J.~Yee, \emph{Convex function approximations for {M}arkov decision processes},
  arXiv:1712.00970 (Preprint).

\end{thebibliography}
\bibliographystyle{amsplain}

\end{document}